\documentclass[12pt]{amsart}

\usepackage{amssymb,graphicx,amsthm}
\usepackage{amsmath}
\usepackage{fancyvrb}
\usepackage{graphicx}

\def\S{{\mathcal S}}

\def\V{{\mathcal V}}

\counterwithin{figure}{section}

\newtheorem{thm}{Theorem}[section]

\newtheorem{prop}[thm]{Proposition}

\begin{document}

\title{Untouchable sets of size $2q \pm 1$ in $PG(2,q)$}
\author{Jeremy M. Dover}
\address{1204 W. Yacht Dr., Oak Island, NC 28465 USA}
\email{\tt doverjm@gmail.com}

\begin{abstract}
An untouchable set in a projective plane is a set of points such that no line of the plane meets the set in exactly one point. Recently, H\'eger and Nagy (Avoiding Secants of Given Size in Finite Projective Planes, {\em J. Combin. Des.} 33:83--93, 2024.) provided a generalization of untouchable sets to $k$-avoiding sets, and addressed the issue of the spectrum of sizes that such sets can attain in finite planes. Specific to the untouchable set case, the authors state as an open question the existence of untouchable sets of size $2q-1$ and $2q+1$. We answer this question in the affirmative for Desarguesian planes of even order, and provide a construction of untouchable sets of size $2q+1$ in $PG(2,q)$ for $ q \equiv 3\pmod{4}$.
\end{abstract}

\maketitle

\section{Introduction}
An {\em untouchable set}, also known as a {\em set without tangents}, in a projective plane is a set of points such that no line of the plane meets the set in exactly one point. Such sets seem to have first been addressed by Blokhuis, Seress and Wilbrink~\cite{blsewi}, wherein the authors prove that the smallest such set in the Desarguesian plane with $q$ odd is $q + \frac14 \sqrt{2q} + 2$. Moreover, the authors provide an example of size $2q-2$ for all such planes of order at least 7.

The minimum size in $PG(2,q)$ for $q$ even is clearly $q+2$; an untouchable set must have at least $q+2$ points (fix one point, and there is at least 1 more point on each of the $q+1$ lines through the fixed point), and a hyperoval is a set of $q+2$ points which lines meet in only $0$ or $2$ points. Numerous examples of untouchable sets in $PG(2,q)$ for $q$ even are known; the previously studied sets of {\em even type}, where every line meets the set in an even number of points, are examples of untouchable sets. Blokhuis, Sz\H{o}nyi and Weiner~\cite{blszwe} show that if an untouchable set in $PG(2,q)$ with $q$ even is not a set of even type, then the set must have at least $q + 1 + \sqrt{q/6}$ points. The authors also comment that the union of two non-disjoint hyperovals is an untouchable set which is not of even type.

The following, almost trivial, observation is related, but does not seem to have been explicitly stated previously.

\begin{prop}
Let $S$ and $T$ be untouchable sets in a projective plane $\pi$. Then $S \cup T$ is an untouchable set.
\end{prop}

\begin{proof}
Let $\ell$ be a line of $\pi$, and suppose $\ell$ contains a point $P$ of $S \cup T$. Then $P \in S$ or $P \in T$, possibly both, in which case $\ell$ meets either $S$ or $T$ in at least one point. Since $S$ and $T$ are untouchable sets, this means $\ell$ meets $S$ or $T$ in at least 2 points, meaning $\ell$ cannot meet $S \cup T$ in a single point.
\end{proof}

H\'eger and Nagy~\cite{hn} generalize the concept of an untouchable set to a {\em $k$-avoiding set}, where no line can meet the set in exactly $k$ points. In this work, the authors provide a number of constructions showing that for most values of $k$, there is a $k$-avoiding set of any size; the exceptions to this are values of $k$ near the extremes, where untouchable sets lie. For untouchable sets, the authors show that in $PG(2,q)$ for $q \ge 4$, untouchable sets of all sizes in ${2q} \cup [2q+2,q^2+q+1]$ exist. Specifically, the authors leave as an open question the existence of untouchable sets of size $2q-1$ and $2q+1$ in $PG(2,q)$, noting that their computational results suggest that untouchable sets of size $2q-1$ exist for $q \ge 8$ even, but not for $q$ odd, while untouchable sets of size $2q+1$ exist for all sufficiently large $q$.

In this article, we provide constructions which show the existence of untouchable sets of size $2q-1$ and $2q+1$ in $PG(2,q)$ for all even $q \ge 8$. We also provide a construction of untouchable sets of size $2q+1$ in $PG(2,q)$ when $q \equiv 3 \pmod{4}$, $q \ge 7$.

\section{Constructions for $q$ even}
In this section we provide constructions for untouchable sets of size $2q-1$ and $2q+1$ in $PG(2,q)$ for $q \ge 8$ even. Our constructions are based on pencils of conics in this plane; the possible pencils are known and reported in Hirschfeld~\cite[Table 7.7]{hirschfeld}. Modelling $PG(2,q)$ as a three-dimensional vector space over $GF(q)$ with homogeneous coordinates $x$, $y$ and $z$, we are interested in the pencil of conics generated by the conics with equations $xy$ and $z^2+yz+xz$. For $k \in GF(q)$, we define $C_k$ to be the conic $\V(kxy+ z^2+yz+xz)$, the set of zeroes of the associated quadratic. From Hirschfeld, we know that three of these conics are degenerate line pairs, namely $\V(xy)$, $C_0 = \V(z^2+yz+xz) = \V(z(x+y+z))$, and $C_1 = \V(xy+z^2+yz+xz) = \V((x+z)(y+z))$, while the remainder are non-degenerate conics. Moreover, all conics in the pencil share four points, namely $(1,0,0),(0,1,0),(1,0,1),(0,1,1)$. For brevity, in what follows we will use the term conic to mean non-degenerate conic, and will specify other types as a line pair or repeated line as needed.

A conic in $PG(2,q)$ has $q+1$ points, no three collinear, and at each point there is a unique tangent line. When $q$ is even, these tangents are all concurrent at a point called the {\em nucleus}, which for $C_k$ is the point $(1,1,k)$, see Hirschfeld~\cite[Corollary 7.12]{hirschfeld}.

Adding the nucleus to the conic creates a {\em hyperconic}, a special case of a hyperoval. Thus we have a set of conics which we can use to create hyperovals by adding their nuclei, and whose union will form an untouchable set. Letting $a \in GF(q)$ with $a \neq 0,1$, the nucleus of $C_a$ is $(1,1,a)$. Every point not in the intersection of all elements of the pencil must lie on a unique element of the pencil, and the nucleus of $C_a$ cannot lie on any of the line pairs in the pencil, since that would force the nucleus to lie on a line which makes 2-point contact with $C_a$. Indeed, algebraically it is easy to see that $(1,1,a)$ lies on $C_{a^2}$. These facts let us prove:

\begin{thm}
In $PG(2,q)$ with $q \ge 8$, there exists an untouchable set of size $2q-1$. Moreover, if $q$ is an even power of 2 there exists an untouchable set of size $2q-2$.
\end{thm}

\begin{proof}
Let $a \in GF(q)$ with $a \ne 0,1$. If $a$ is not a cube root of unity, then the nucleus of $C_a$ lies on $C_{a^2}$, and the nucleus of $C_{a^2}$ lies on $C_{a^4} \ne C_a$. Define $\S = C_{a} \cup C_{a^2} \cup \{(1,1,a^2)\}$. $C_a$ has $q+1$ points, and $C_{a^2}$ provides an additional $(q+1) - 4 = q-3$ points. As $(1,1,a^2)$, the nucleus of $C_{a^2}$, lies on neither $C_a$ nor $C_{a^2}$, $\S$ has $q+1 + q-3 + 1 = 2q-1$ points. On the other hand $\S$ is the union of the hyperconics obtained from augmenting $C_a$ and $C_{a^2}$ by their nuclei. Therefore $\S$ is an untouchable set.

If $a$ is a cube root of unity, which can only occur when $q$ is an even power of 2, $a^4 = a$, so the nucleus of $C_a$ lies on $C_{a^2}$ and the nucleus of $C_{a^2}$ lies on $C_a$, so $\S = C_{a} \cup C_{a^2}$ has size $2q-2$ and is the union of two hyperconics, making it an untouchable set of size $2q-2$.
\end{proof}

To construct untouchable sets of size $2q+1$ in $PG(2,q)$ for $q$ even, we can appeal to a different conic pencil, generated by $xy$ and $z^2+xz$. From Hirschfeld~\cite[Table 7.7]{hirschfeld}, we see this pencil has two line pairs, namely the generators, and $q-1$ conics. For $k \in GF(q)$, $k \ne 0$, the conics are of the form $D_k = \V(kxy + z^2 + xz)$. All of the elements in the pencil contain three common points, namely $(1,0,0),(0,1,0),(1,0,1)$, and as before the nucleus of $D_k = (0,1,k)$. Unlike before, there is no combinatorial requirement that the nucleus cannot lie on one of the line pairs, and in fact this is the case; the nuclei of all conics in the pencil lie on $\V(xy)$. This allows us to prove:

\begin{thm}
In $PG(2,q)$ with $q \ge 8$, there exists an untouchable set of size $2q+1$.
\end{thm}

\begin{proof}
Let $a,b \in GF(q)$ with $a,b \ne 0 $ and $a \ne b$. Let $\S = D_a \cup D_b \cup \{(0,1,a),(0,1,b)\}$. $D_a$ has $q+1$ points, and $D_b$ shares three points with $D_a$, so $D_a \cup D_b$ has $2q-1$ points. The points $(0,1,a)$ and $(0,1,b)$ are the nuclei of $D_a$ and $D_b$, respectively, so neither lies on either $D_a$ or $D_b$ and $\S$ has $2q+1$ points. As $\S$ is the union of the hyperconics $D_a \cup \{(0,1,a)\}$ and $D_b \cup \{(0,1,b)\}$, it is the untouchable set we require.
\end{proof}

As two distinct conics cannot meet in more than 4 points, there is no hope of using conic pencils to fill any other gaps in the spectrum of sizes of untouchable sets below $2q-2$, and we have seen the only case where size $2q-2$ is achievable, where we found two conics intersecting in 4 points which each contained the other's nucleus. However, the union construction works for any pair of hyperovals, not just hyperconics. Indeed Blokhuis, Sz\H{o}nyi and Weiner~\cite{blszwe} construct an untouchable set in $PG(2,16)$ of size $2q-4$ by taking the union of a hyperconic and a Lunelli-Sce hyperoval which intersect each other in 8 points. 

\section{Construction for $q$ odd}

When $q$ is odd, a finite projective plane of order $q$ cannot contain a hyperoval, and in general less seems known about untouchable sets. However, pencils of conics still have a role to play, in this case the pencil generated by $xy$ and $z^2$. This pencil was studied by Dover and Mellinger~\cite{dm} in connection with semiovals, namely point sets that have a unique tangent line at each point. This pencil consists of a set of $q-1$ conics which all contain the same set of two points, the repeated line $z^2$ which contains the two common points, and the line pair $xy$, which is the union of the tangent lines to each of the conics at the common points. We label the conics in the pencil $C_k = \V(xy+kz^2)$ for $k \in GF(q)$, $k \ne 0$.

The following relationship between conics of this pencil is essentially given in Dover and Mellinger~\cite[Lemma 3.2, Theorem 4.1]{dm}, though the language is slightly different.
\begin{prop}
\label{mutextcond}
In $PG(2,q)$ with $q$ odd, let $C_a$ and $C_b$ be distinct non-degenerate conics. Then every point on $C_b$ not lying on $C_a$ is either exterior to $C_a$ or interior to $C_a$ as $a(a-b)$ is a square or non-square in $GF(q)$.
\end{prop}

In order to form an untouchable set from a conic with $q$ odd, we need to add exterior points to additionally cover tangents. Thus it makes sense to look at conics that are {\em mutually exterior}, i.e., the points on exactly one of the conics are exterior to the other conic. The following Proposition shows that such a pair of conics exists when $q \ge 5$.

\begin{prop}
In $PG(2,q)$ with $q \ge 7$ odd, there exists a pair of mutually exterior conics $C_a$ and $C_b$.
\end{prop}

\begin{proof}
Without loss of generality, we assume $a = 1$. Using Proposition~\ref{mutextcond}, we find that $C_1$ and $C_b$ are mutually exterior if and only if $1-b$ and $b(b-1)$ are non-zero squares in $GF(q)$. If $q \equiv 1 \pmod{4}$, then $-1$ is a square, and these conditions are equivalent to both $b$ and $b-1$ being squares, necessarily non-zero as $b \ne 1$. From Dickson~\cite{dickson}, we know that the number of non-zero squares $b-1$ such that $(b-1)+1 = b$ is also a non-zero square is $\frac14(q-5)$, so a pair of mutually exterior conics exists for $q \equiv 1 \pmod{4}$ as long as $q \ge 9$.

 If $q \equiv 3 \pmod{4}$, then $-1$ is a non-square, and our conditions are equivalent to both $b$ and $b-1$ being non-squares. Again from Dickson~\cite{dickson}, the number of squares $b-1$ such that $b$ is a non-square is $\frac14(q+1)$, so the number of non-squares $b-1$ such that $b$ is a non-square is $\frac14(q-3)$, which is greater than zero for $q \ge 7$.
\end{proof}

We are now ready to exhibit an untouchable set of size $2q+1$, when $q \equiv 3 \pmod{4}$.

\begin{thm}
\label{3mod4}
In $PG(2,q)$ for $q \equiv 3 \pmod{4}$ odd with $q \ge 7$, let $b$ be a non-square such that $b-1$ is also a non-square, so that $C_1$ and $C_b$ are a pair of mutually exterior conics. Then $\S = C_1 \cup C_b \cup \{(0,0,1)\}$ is an untouchable set of size $2q+1$.
\end{thm}

\begin{proof}
First note that $\S$ has size $2q+1$: $C_1$ has $q+1$ points; $C_b$ adds $q-1$ additional points, since $C_1$ and $C_b$ intersect in 2 points; and $(0,0,1)$ does not lie on either conic.

Count the number of flags $(P,\ell)$, where $P$ is a point of $C_b$ not also in $C_1$, and $\ell$ is a tangent line to $C_1$ at a point not also in $C_b$. Each of the $q-1$ such points $P$ is an exterior point to $C_1$, and thus lies on two tangents to $C_1$. Moreover, these tangents cannot pass through the points of $C_1 \cap C_b$, since those lines are also tangent to $C_b$. Thus there are $2(q-1)$ such flags. On the other hand, there are $q-1$ tangents to $C_1$ at points not also in $C_b$, and each meets $C_b$ in at most 2 points, so the number of flags under consideration is at most $2(q-1)$, with equality only if every tangent to $C_1$ at a point not also in $C_b$ is a secant to $C_b$. The same argument works with the roles of $C_1$ and $C_b$ exchanged, so any line that meets $\S$ in a point of $C_1 \triangle C_b$ necessarily meets $\S$ in at least 2 points.

A line $\ell$ containing either of the points of $C_1 \cup C_b$ is either a secant to both $C_1$ and $C_b$, or is tangent to both. But the tangents at these points are concurrent at $(0,0,1)$, so the line $\ell$ meets $\S$ in at least 2 points.

Finally let $\ell$ be a line through $(0,0,1)$. The lines with coordinates $[1,0,0]$ and $[0,1,0]$ are the tangents to $C_1$ and $C_b$ at their common points, and thus contain 2 points of $\S$. So we can restrict our attention to $\ell = [1,c,0]$ for some $c \in GF(q)$, $c \ne 0$. If $c$ is a square, then there exists $\gamma \in GF(q)$ such that $\gamma^2 = c$, and $(c,-1,\gamma)$ is a point on $C_1$. If $c$ is a non-square, there exists $\delta \in GF(q)$ such that $\delta^2 = \frac{c}{b}$, as both $c$ and $b$ are non-squares; in this case, $(c,-1,\delta)$ is a point on $C_b$. Thus $\ell$ meets one of $C_1$ or $C_b$ in at least one point, and $\ell$ is not a tangent to $\S$. This implies $\S$ is an untouchable set.
\end{proof}

The proof of Theorem~\ref{3mod4} almost works for the $q \equiv 1 \pmod{4}$ case as well, except that it fails in the last step; there will be tangent lines through $(0,0,1)$. Moreover, this technique does not seem reparable, as the point $(0,0,1)$ is needed to ensure the lines of the line pair are not tangents to $\S$.

\section{Conclusion}
The constructions of this paper provide corroborating evidence for the conjectures made by H\'eger and Nagy~\cite{hn} regarding the existence of untouchable sets of size $2q \pm 1$; generally we have exhibited an infinite family of examples where they had computational examples in small planes, and not found counterexamples to their suspected non-existence results. The main place we have failed is in finding untouchable sets of size $2q+1$ in $PG(2,q)$ for $q \equiv 1 \pmod{4}$; we suspect these sets exist, but do not believe that conics are involved. At the very least, we conjecture that when $q \equiv 1 \pmod{4}$, an untouchable set of size $2q+1$ in $PG(2,q)$ may not contain a conic. The non-existence of untouchable sets of size $2q-1$ in $PG(2,q)$ for all odd $q$ remains open, but we suspect that is a significantly harder problem, given how far $2q+1$ is from the general lower bound of $q + \frac14 \sqrt{2q} + 2$.

\bibliographystyle{plain}

\end{document}